\documentclass[a4paper,12pt]{amsart}
\usepackage[margin=1in]{geometry}

\usepackage{amsmath,amsfonts,amssymb}
\usepackage{bm}
\usepackage{amsthm}
\usepackage{tikz}
\usetikzlibrary{arrows.meta}
\newcounter{num}
\newcommand{\Rnum}[1]{\setcounter{num}{#1}\Roman{num}}
\newcommand{\rnum}[1]{\setcounter{num}{#1}\roman{num}}
\numberwithin{equation}{section}

\newcommand{\Spec}{\mathrm{Spec}}
\newcommand{\et}{\text{ét}}
\newcommand{\intl}{\text{int}}
\newcommand{\satd}{\text{sat}}
\newcommand{\gp}{\text{gp}}
\newcommand{\cok}{\mathrm{Cok}}
\newcommand{\ext}{\mathrm{Ext}}
\newcommand{\rank}{\mathrm{rank}}
\newcommand{\can}{\mathrm{can}}
\newcommand{\bLog}{\mathbf{Log}}
\newcommand{\id}{\mathrm{id}}
\newcommand{\ul}[1]{\underline{#1}}

\theoremstyle{plain}
\newtheorem{thm}{Theorem}[section]
\newtheorem{lem}[thm]{Lemma}
\newtheorem{prop}[thm]{Proposition}
\newtheorem{cor}[thm]{Corollary}

\theoremstyle{definition}

\newtheorem{dfn}[thm]{Definition}

\theoremstyle{remark}
\newtheorem{rem}[thm]{Remark}

\theoremstyle{plain}
\newtheorem*{thm*}{Theorem}
\newtheorem*{lem*}{Lemma}
\newtheorem*{prop*}{Proposition}
\newtheorem*{cor*}{Corollary}
\newtheorem*{conj*}{Conjecture}
\newtheorem*{mthm*}{Main Theorem}

\theoremstyle{definition}
\newtheorem*{ass*}{Assumption}
\newtheorem*{dfn*}{Definition}

\theoremstyle{remark}
\newtheorem*{rem*}{Remark}

\begin{document}

\title{Logarithmic Kunz and Radu--André Theorems}
\author{Ryoma Takeuchi}
\address{Department of Mathematics, Institute of Science Tokyo, 2-12-1 Ookayama, Meguro, Tokyo 152-8551, Japan}
\email{takeuchi.r.1627@m.isct.ac.jp}
\begin{abstract}
    Frobenius morphisms play a crucial role in algebraic geometry and commutative algebra in positive characteristic.
    In particular, \textit{the Kunz theorem}, characterizing non-singularity via Frobenius morphisms, is a monumental work and has a relative version, called \textit{the Radu--André theorem}.

    In this paper, we extend the Kunz theorem and the Radu--André theorem to logarithmic algebraic geometry in positive characteristic.
\end{abstract}
\maketitle
\tableofcontents

\section{Introduction}

Let $p$ be a prime number. 
For a scheme $X$ over $\mathbb{F}_{p}$, we denote the absolute Frobenius morphism on $X$ by $F_{X}\colon X\to X$.
Frobenius morphisms play an essential role in algebraic geometry and commutative algebra in positive characteristic.
The following Kunz theorem is a monumental result characterizing regularity via Frobenius morphisms.
\begin{thm}[{\cite{MR252389}}]
    A locally Noetherian scheme $X$ over $\mathbb{F}_{p}$ is regular if and only if the absolute Frobenius morphism $F_{X}\colon X\to X$ is flat.
\end{thm}

Let $f\colon X\to Y$ be a morphism of schemes over $\mathbb{F}_{p}$.
Since we have $F_{Y}\circ f=f\circ F_{X}$, there is a unique morphism $F_{X/Y}\colon X\to X^{(p)}:=X\times_{Y,F_{Y}}Y$ fitting into the following commutative diagram
\begin{equation*}
    \begin{tikzpicture}[auto]
        \node (a) at (-2,2) {$X$};
        \node (b) at (0,0) {$X$};
        \node (c) at (2,2) {$Y$};
        \node (d) at (2,0) {$Y,$};
        \node (F) at (0,2) {$X^{(p)}$};
        \draw[->] (a) to[bend right=30] node[swap] {$\scriptstyle F_{X}$} (b);
        \draw[->] (a) to[bend left=60] node {$\scriptstyle f$} (c);
        \draw[->] (b) to node[swap] {$\scriptstyle f$} (d);
        \draw[->] (c) to node {$\scriptstyle F_{Y}$} (d);
        \draw[->] (F) to (b);
        \draw[->] (F) to (c);
        \draw[->, dashed] (a) to node {$\scriptstyle F_{X/Y}$} (F);
    \end{tikzpicture}
\end{equation*}
where the square is cartesian.
The morphism $F_{X/Y}$ is called \textit{the relative Frobenius morphism} or \textit{the Radu--André morphism} of $f$.
The following Radu--André theorem is a relative version of the Kunz theorem.
\begin{thm}[{\cite{MR1172271}, \cite{MR1214408}}]\label{thm:RA_classical}
    A morphism $f\colon X\to Y$ of locally Noetherian schemes over $\mathbb{F}_{p}$ is regular if and only if the relative Frobenius morphism $F_{X/Y}\colon X\to X^{(p)}$ is flat.
\end{thm}

In this paper, we aim to extend the Radu--André theorem to morphisms of fine log schemes.
From the paper \cite{MR1463703} of K. Kato, log geometry has been broadly studied.
In \cite{MR2032986}, Olsson introduced log stacks defined for fine log schemes and used them to give nice characterizations of log smoothness, log étaleness, and log flatness which were defined by Kato.
Recently, Molcho and Temkin defined and studied log regularity of morphisms of fine log schemes using Olsson's log stacks (\cite{MR4211089}).
Following these works, we can give a log analogue of the Radu--André theorem.
Precisely, we have the following:
\begin{thm}[{Theorem \ref{thm:logRA}}]
    Let $f\colon X\to Y$ be an integral morphism of locally Noetherian fine log schemes over $\mathbb{F}_{p}$ such that the relative characteristic monoid $\mathcal{M}_{X/Y}^{\gp}$ is $p$-torsion free.
    Then $f$ is log regular if and only if $F_{X/Y}\colon X\to X^{(p)}$ is log flat.
\end{thm}
As a consequence, we obtain a log analogue of the Kunz theorem:
\begin{thm}[{Corollary \ref{cor:logKunz}}]
    Let $X$ be a locally Noetherian fine log scheme over $\mathbb{F}_{p}$ such that the sharpened characteristic monoid $\overline{\mathcal{M}}_{X}^{\gp}$ is $p$-torsion free.
    Then $X\to \Spec(\mathbb{F}_{p})$ is log regular if and only if $F_{X}\colon X\to X$ is log flat.
    Here, we endow $\Spec(\mathbb{F}_{p})$ with the trivial log structure.
    
    In particular, if $X$ is saturated, then $X$ is log regular if and only if $F_{X}$ is log flat.
\end{thm}

This paper is organized as follows:
In Section 2, we recall basic notions of log geometry.
In Section 3, we prove Theorem \ref{thm:logRA} and Corollary \ref{cor:logKunz}.

\vspace{5pt}

\noindent
\textbf{Acknowledgements}.
The author would like to thank Yuri Yatagawa for her valuable discussions.
The author is also grateful to Kazuki Hayashi, Rirai Ikeda, and Kanau Shimada for their comments.
The author was supported by JST SPRING, Japan Grant Number JPMJSP2180.

\section{Preliminaries on log geometry}

\subsection{Log schemes}

\indent

Here, we recall basic notions about log schemes.
Our main references are \cite{MR1463703} and \cite{MR3838359}.

\subsubsection{Convention}

\indent

\begin{itemize}
    \item[(1)] We assume that every monoid is unital and commutative.
    \item[(2)] For a monoid $P$, we write the monoid operation multiplicatively. Hence we write $1$ for the unity of $P$.
    \item[(3)] For a scheme $X$, we consider sheaves of monoids on the étale site $X_{\et}$ if there is no confusion.
\end{itemize}

\subsubsection{Monoids}

\indent

\begin{itemize}
    \item[(1)] A monoid $P$ is \textit{integral} if the canonical morphism $P\to P^{\gp}$ is injective, where $P^{\gp}$ is the Grothendieck group of $P$.
    \item[(2)] A monoid $P$ is \textit{fine} if it is finitely generated and integral.
    \item[(3)] Let $n$ be a positive integer. An integral monoid $P$ is \textit{$n$-saturated} if $x^n\in P$ implies $x\in P$ for any $x\in P^{\gp}$.
    \item[(4)] An integral monoid $P$ is \textit{saturated} if it is $n$-saturated for every $n\geq 1$ (or equivalently, $p$-saturated for every prime number $p$).
    \item[(5)] A monoid $P$ is \textit{sharp} if $P^{\times}=1$.
        For a monoid $P$, its \textit{sharpening} is $\overline{P}:=P/P^{\times}$.
\end{itemize}

\subsubsection{Log structures}

\indent

\begin{itemize}
    \item[(1)] A \textit{prelog structure} on a scheme $X$ is a morphism $\alpha\colon \mathcal{M}\to \mathcal{O}_{X}$ of sheaves of monoids, where we regard $\mathcal{O}_{X}$ as a sheaf of monoids by its multiplication.
    \item[(2)] A \textit{log structure} on a scheme $X$ is a prelog structure $\alpha\colon \mathcal{M}\to \mathcal{O}_{X}$ which induces
        \begin{equation*}
            \alpha^{-1}(\mathcal{O}_{X}^{\times})\stackrel{\cong}{\to} \mathcal{O}_{X}^{\times}.
        \end{equation*}
        By this isomorphism, we can identify $\mathcal{O}_{X}^{\times}$ as a subsheaf of $\mathcal{M}$.
        The \textit{trivial log structure} on $X$ is the inclusion $\mathcal{O}_{X}^{\times}\subset \mathcal{O}_{X}$.
    \item[(3)] For a prelog structure $\alpha\colon \mathcal{M}\to \mathcal{O}_{X}$, the \textit{associated log structure} $\alpha^{\log}\colon \mathcal{M}^{\log}\to \mathcal{O}_{X}$ is defined by the following diagram
        \begin{equation*}
            \begin{tikzpicture}[auto]
                \node (a) at (0,2) {$\alpha^{-1}(\mathcal{O}_{X}^{\times})$};
                \node (b) at (0,0) {$\mathcal{M}$};
                \node (c) at (2,2) {$\mathcal{O}_{X}^{\times}$};
                \node (d) at (2,0) {$\mathcal{M}^{\log}$};
                \node (F) at (4,0) {$\mathcal{O}_{X},$};
                \draw[{Hooks[right]}->] (a) to (b);
                \draw[->] (a) to node {$\scriptstyle \alpha$} (c);
                \draw[->] (b) to (d);
                \draw[->] (c) to (d);
                \draw[->] (b) to[bend right=60] node[swap] {$\scriptstyle \alpha$} (F);
                \draw[{Hooks[right]}->] (c) to[bend left=30] (F);
                \draw[->, dashed] (d) to node[swap] {$\scriptstyle \alpha^{\log}$} (F);
            \end{tikzpicture}
        \end{equation*}
        where the square is cocartesian.
    \item[(4)] A \textit{log scheme} is a scheme $X$ endowed with a log structure $\alpha\colon \mathcal{M}_{X}\to \mathcal{O}_{X}$.
    \item[(5)] A log scheme $X$ is \textit{integral} if $\mathcal{M}_{X,\bar{x}}$ is integral for any $x\in X$.
    \item[(6)] For a positive integer $n$, an integral log scheme $X$ is \textit{$n$-saturated} if $\overline{\mathcal{M}}_{X,\bar{x}}$ is $n$-saturated for any $x\in X$.
    \item[(7)] An integral log scheme $X$ is \textit{saturated} if it is $n$-saturated for every $n\geq 1$ (or equivalently, $p$-saturated for every prime number $p$).
    \item[(8)] A \textit{morphism} of log schemes is a pair $(f,f^{\flat})$ where $f\colon X\to Y$ is a morphism of underlying schemes and $f^{\flat}\colon \mathcal{M}_{Y}\to f_{*,\log}\mathcal{M}_{X}$ is a morphism of log structures on $Y$.
        Here, $f_{*,\log}\mathcal{M}_{X}$ is the pullback of $f_{*}\mathcal{M}_{X}\to f_{*}\mathcal{O}_{X}$ along $\mathcal{O}_{Y}\to f_{*}\mathcal{O}_{X}$.
        We simply write $f_{*}\mathcal{M}_{X}$ for $f_{*,\log}\mathcal{M}_{X}$ and $f$ for $(f,f^{\flat})$.
    \item[(9)] A morphism $f\colon X\to Y$ of log schemes is \textit{strict} if $f^{\flat}\colon \mathcal{M}_{Y}\to f_{*}\mathcal{M}_{X}$ induces $f^{*}\mathcal{M}_{Y}\stackrel{\cong}{\to} \mathcal{M}_{X}$, where $f^{*}\mathcal{M}_{Y}$ is the log structure on $X$ associated to the prelog structure $f^{-1}\mathcal{M}_{Y}\to f^{-1}\mathcal{O}_{Y}\to \mathcal{O}_{X}$.
\end{itemize}

\subsubsection{Charts}

\indent

\begin{itemize}
    \item[(1)] For a log scheme $X$, a \textit{chart} for $X$ subordinate to a monoid $P$ is a morphism $P\to \Gamma(X,\mathcal{M}_{X})$ of monoids which induces
        \begin{equation*}
            P_{X}^{\log}\stackrel{\cong}{\to} \mathcal{M}_{X},
        \end{equation*}
        where $P_{X}$ denotes the constant sheaf of monoids associated to $P$.
        For a log scheme $X$ over a base ring $R$ (for instance, $R=\mathbb{Z}$), giving a chart $P\to \Gamma(X,\mathcal{M}_X)$ for $X$ is equivalent to giving a strict morphism $X\to \mathsf{A}_{P}$ of log schemes over $R$, where $\mathsf{A}_{P}:=\Spec(R[P])$ is endowed with a canonical log structure induced by $P\to R[P]$.
        A chart $P\to \Gamma(X,\mathcal{M}_{X})$ is said to be \textit{fine} if $P$ is fine.
    \item[(2)] A log scheme $X$ is \textit{quasi-coherent} (resp.\ \textit{coherent}) if, étale locally, $X$ admits charts (resp.\ charts subordinate to finitely generated monoids).
        A log scheme $X$ is \textit{fine} if it is coherent and integral.
    \item[(3)] For a morphism $f\colon X\to Y$ of log schemes, a \textit{chart} for $f$ is a triple
        \begin{equation*}
            (a\colon Y\to \mathsf{A}_{P},\theta\colon P\to Q,b\colon X\to \mathsf{A}_{Q})
        \end{equation*}
        such that $a$ and $b$ are strict and that the following diagram is commutative:
        \begin{equation*}
            \begin{tikzpicture}[auto]
                \node (x) at (0,2) {$X$};
                \node (y) at (0,0) {$Y$};
                \node (q) at (2,2) {$\mathsf{A}_{Q}$};
                \node (p) at (2,0) {$\mathsf{A}_{P}$};
                \draw[->] (x) to node[swap] {$\scriptstyle f$} (y);
                \draw[->] (x) to node {$\scriptstyle b$} (q);
                \draw[->] (y) to node {$\scriptstyle a$} (p);
                \draw[->] (q) to node {$\scriptstyle \mathsf{A}_{\theta}$} (p);
            \end{tikzpicture}
        \end{equation*}
\end{itemize}

Next, we recall a useful class of charts for morphisms of fine log schemes, \textit{neat charts}.
\begin{itemize}
    \item[(4)] A fine chart $P\to \Gamma(Y,\mathcal{M}_{Y})$ for a fine log scheme $Y$ is \textit{neat} at $y\in Y$ if the induced morphism $P^{\gp}\to \overline{\mathcal{M}}_{Y,\bar{y}}^{\gp}$ is an isomorphism, where $\overline{\mathcal{M}}_{Y}:=\mathcal{M}_{Y}/\mathcal{O}_{Y}^{\times}$.
    \item[(5)] A fine chart $(a,\theta,b)$ for a morphism $f\colon X\to Y$ of fine log schemes is \textit{neat} at $x\in X$ if the following are satisfied:
        \begin{itemize}
            \item[(\rnum{1})] $\theta^{\gp}\colon P^{\gp}\to Q^{\gp}$ is injective.
            \item[(\rnum{2})] The induced morphism $\cok(\theta^{\gp})\to \mathcal{M}_{X/Y,\bar{x}}^{\gp}$ is an isomorphism, where $\mathcal{M}_{X/Y}$ is the sheaf $\cok(f^{*}\mathcal{M}_{Y}\to \mathcal{M}_{X})$.
                (It can be easily checked that we have $\mathcal{M}_{X/Y,\bar{x}}=\cok((f^{*}\mathcal{M}_{Y})_{\bar{x}}\to \mathcal{M}_{X,\bar{x}})=\cok(\overline{\mathcal{M}}_{Y,f(\bar{x})}\to \overline{\mathcal{M}}_{X,\bar{x}})$.) 
        \end{itemize}
\end{itemize}
The conditions for étale local existence of neat charts are described as follows:

\begin{prop}[{\cite[Chapter \Rnum{3}, Theorem 1.2.7]{MR3838359}}]\label{prop:exist_neat}
    Let $f\colon X\to Y$ be a morphism of fine log schemes.
    Write $y=f(x)$ for $x\in X$.
    \begin{itemize}
        \item[\textup{(1)}] If the sequence
            \begin{equation*}
                0\to \mathcal{O}_{Y,\bar{y}}^{\times}\to \mathcal{M}_{Y,\bar{y}}^{\gp}\to \overline{\mathcal{M}}_{Y,\bar{y}}^{\gp}\to 0
            \end{equation*}
            splits, then, étale locally around $y$, there exists a fine chart which is neat at $y$.
            In particular, if the order of the torsion subgroup of $\overline{\mathcal{M}}_{Y,\bar{y}}^{\gp}$ is invertible in $k(\bar{y})$, then, étale locally around $y$, there exists a fine chart which is neat at $y$.
        \item[\textup{(2)}] Let $a\colon Y\to \mathsf{A}_{P}$ be a fine chart around $y$.
            If we have $\ext^{1}(\mathcal{M}_{X/Y,\bar{x}}^{\gp},\mathcal{O}_{X,\bar{x}}^{\times})=0$, then, étale locally around $x$ and $y$, there exists a fine chart $(a,\theta,b)$ which is neat at $x$.
            In particular, if the order of the torsion subgroup of $\mathcal{M}_{X/Y,\bar{x}}^{\gp}$ is invertible in $k(\bar{x})$, then, étale locally around $x$ and $y$, there exists a fine chart $(a,\theta,b) $which is neat at $x$.
    \end{itemize}
\end{prop}

\subsection{Relative integrality and saturatedness}

\indent

Here, we discuss fibered products of log schemes.

Let $f\colon X\to S$ and $g\colon Y\to S$ be morphisms of log schemes.
The fibered product $Z=X\times_{S}Y$ taken in the category of log schemes is given by the fibered product $\ul{Z}=\ul{X}\times_{\ul{S}}\ul{Y}$ of the underlying schemes with the log structure
\begin{equation*}
    p^{*}\mathcal{M}_{X}\oplus_{h^{*}\mathcal{M}_{S}}q^{*}\mathcal{M}_{Y}\to \mathcal{O}_{\ul{Z}}
\end{equation*}
induced by $p^{*}\mathcal{M}_{X}\to \mathcal{O}_{\ul{Z}}$ and $q^{*}\mathcal{M}_{Y}\to \mathcal{O}_{\ul{Z}}$, where $p$ (resp.\ $q$) denotes the projection $\ul{Z}\to \ul{X}$ (resp.\ $\ul{Z}\to \ul{Y}$) and $h=fp=qg$.

The problem is that the fibered product $Z=X\times_{S} Y$ taken in the category of log schemes is not integral (resp.\ saturated) in general even if $X$, $Y$, and $S$ are so.
Then we have to take the associated integral (resp.\ saturated) log scheme $Z^{\intl}$ (resp.\ $Z^{\satd}$) for the fibered product in the category of integral (resp.\ saturated) log schemes.
(For details, see \cite[Chapter \Rnum{3}, Section 2.1]{MR3838359}.)

However, if either $X\to S$ or $Y\to S$ is \textit{integral} (resp.\ \textit{saturated}), then these fibered products coincide.
In this subsection, we recall integrality and saturatedness of morphisms.
Our main references are \cite{MR3838359} and \cite{MR3907739}.

\subsubsection{Integral and saturated morphisms of monoids}

\indent

\begin{itemize}
    \item[(1)] A morphism $P\to Q$ of integral monoids is \textit{integral} if, for any morphism $P\to P'$, the pushout $Q\oplus_{P}P'$ taken in the category of monoids is integral.
    \item[(2)] For a positive integer $n$, an integral morphism $P\to Q$ of integral monoids is \textit{$n$-saturated} if, for any morphism $P\to P'$ such that $P'$ is $n$-saturated, the pushout $Q\oplus_{P}P'$ taken in the category of monoids is $n$-saturated.
    \item[(3)] An integral morphism $P\to Q$ of integral monoids is \textit{saturated} if it is $n$-saturated for every $n\geq 1$ (or equivalently, $p$-saturated for every prime number $p$).
\end{itemize}

\subsubsection{Integral and saturated morphisms of integral log schemes}

\indent

\begin{itemize}
    \item[(1)] A morphism $X\to Y$ of integral log schemes is \textit{integral} if, for any $x\in X$, the induced morphism $\overline{\mathcal{M}}_{Y,f(\bar{x})}\to \overline{\mathcal{M}}_{X,\bar{x}}$ is integral.
    \item[(2)] For a positive integer $n$, an integral morphism $X\to Y$ of integral log schemes is \textit{$n$-saturated} if, for any $x\in X$, the induced morphism $\overline{\mathcal{M}}_{Y,f(\bar{x})}\to \overline{\mathcal{M}}_{X,\bar{x}}$ is $n$-saturated.
    \item[(3)] An integral morphism $X\to Y$ of interal log schemes is \textit{saturated} if it is $n$-saturated for every $n\geq 1$ (or equivalently, $p$-saturated for every prime number $p$).
\end{itemize}

\begin{prop}[{cf.\ \cite[Chapter \Rnum{3}, Theorem 2.5.5]{MR3838359}}]\label{prop:n-sat}
    For a positive integer $n$, let $f\colon X\to Y$ be an $n$-saturated morphism, where $X$ is $n$-saturated.
    Then $\mathcal{M}_{X/Y}^{\gp}$ is $n$-torsion free.
\end{prop}
\begin{proof}
    The proof can be done in the similar way as the proof of \cite[Chapter \Rnum{1}, Proposition 4.8.11]{MR3838359}.
    Write $\theta\colon P=\overline{\mathcal{M}}_{Y,f(\bar{x})}\to Q=\overline{\mathcal{M}}_{X,\bar{x}}$.
    We prove that $\cok(\theta^{\gp})$ is $n$-torsion free.
    Note that $P/\theta^{-1}(1)\to Q$ is also $n$-saturated by \cite[Chapter \Rnum{1}, Proposition 4.8.8 (2)]{MR3838359} and that $P/\theta^{-1}(1)$ is still sharp.
    Hence, by replacing $P$ by $P/\theta^{-1}(1)$, we may assume that we have $\theta^{-1}(1)=1$.
    The morphism $Q^{\gp}\to Q^{\gp}\oplus_{P^{\gp}}Q^{\gp}$ sending $y$ to $(y,y^{-1})$ induces an injection $\cok(\theta^{\gp})\hookrightarrow Q^{\gp}\oplus_{P^{\gp}}Q^{\gp}$.
    Since $Q\oplus_{P}Q$ is $n$-saturated, we see that the $n$-torsion subgroup $\Gamma$ of $\cok(\theta^{\gp})$ is contained in $(Q\oplus_{P}Q)^{\times}$.
    Then the assertion follows since we have $(Q\oplus_{P}Q)^{\times}=Q^{\times}\oplus_{P^{\times}}Q^{\times}=1$ by \cite[Chapter \Rnum{1}, Proposition 4.1.5 (1)]{MR3838359}
\end{proof}

\begin{prop}
    Let $f\colon X\to Y$ be a morphism of integral log schemes and $n$ be a positive integer.
    \begin{itemize}
        \item[\textup{(1)}] \textup{\cite[Chapter \Rnum{3}, Proposition 2.5.3 (2)]{MR3838359}} If $f$ is integral, then, for any morphism $Y'\to Y$ of integral log schemes, the fibered product $X'=X\times_{Y}Y'$ taken in the category of log schemes is integral.
        \item[\textup{(2)}] \textup{\cite[Proposition \Rnum{2}.2.13 (1)]{MR3907739}} If $f$ is $n$-saturated (resp.\ saturated), then, for any morphism $Y'\to Y$ such that $Y'$ is $n$-saturated (resp.\ saturated), the fibered produt $X'=X\times_{Y}Y'$ taken in the category of log schemes is $n$-saturated (resp.\ saturated).
    \end{itemize}
\end{prop}

\subsection{Olsson's log stacks and properties of morphisms of log schemes}

\indent

For a fine log scheme $X$, Olsson defined the stack $\bLog_{X}$ whose objects are fine log schemes over $X$ and whose morphisms are strict morphisms over $X$ (\cite{MR2032986}).
The stack $\bLog_{X}$ parametrizes fine log schemes over $X$.
For every scheme $f\colon Y_{0}\to \ul{X}$ over the underlying scheme $\ul{X}$ of $X$, we can endow $Y_{0}$ with the fine log structure $f^{*}\mathcal{M}_{X}$, which yields a canonical morphism $\ul{X}\to \bLog_{X}$.
Note that a morphism $f\colon X\to Y$ of fine log schemes defines a natural morphism $\bLog_{f}\colon \bLog_{X}\to \bLog_{Y}$.

\begin{dfn}[{\cite[Definition 4.1]{MR2032986}}]
    Let $f\colon X\to Y$ be a morphism of fine log schemes and $\mathcal{P}$ be a property of representable morphisms of algebraic stacks.
    \begin{itemize}
        \item[(1)] We say that $f$ is \textit{log $\mathcal{P}$} if $\bLog_{f}\colon \bLog_{X}\to \bLog_{Y}$ is $\mathcal{P}$.
        \item[(2)] We say that $f$ is \textit{weakly log $\mathcal{P}$} if the composition $w\bLog_{f}\colon \ul{X}\stackrel{\can}{\to} \bLog_{X}\stackrel{\bLog_{f}}{\to} \bLog_{Y}$ is $\mathcal{P}$.
    \end{itemize}
\end{dfn}

\begin{thm}[{\cite[Theorem 4.6]{MR2032986}}, {\cite[Theorem 5.1.3]{MR4211089}}]\label{thm:logP}
    Let $f\colon X\to Y$ be a morphism of fine log schemes and $\mathcal{P}$ be one of the properties $\{\text{\textup{smooth, étale, flat, regular}}\}$.
    Then $f$ is log $\mathcal{P}$ if and only if it is weakly log $\mathcal{P}$.
\end{thm}

\begin{rem}

    \indent

    \begin{itemize}
        \item[(1)] For smoothness, étaleness, and flatness, the assertions of Theorem \ref{thm:logP} are proved in \cite[Theorem 4.6]{MR2032986}.
            For regularity, the assertion is proved in \cite[Theorem 5.1.3]{MR4211089}.
        \item[(2)] In \cite[Theorem 4.6]{MR2032986}, it is also proved that, for a morphism $f\colon X\to Y$ of fine log schemes, log smoothness (resp.\ log étaleness, resp.\ log flatness) is equivalent to log smoothness (resp.\ log étaleness, resp.\ log flatness) in the original sense of Kato described in \cite{MR1463703}, \cite{MR4342362}.
    \end{itemize}
\end{rem}
In this paper, we mainly consider log flatness and log regularity.

\subsubsection{Log flatness and charts}

\begin{prop}[{\cite[Theorem 4.6 (\rnum{4})]{MR2032986}}]\label{prop:logflat}
    Let $f\colon X\to Y$ be a morphism of fine log schemes.
    Then $f$ is log flat if and only if, fppf locally, there exists a fine chart
    \begin{equation*}
        (a\colon Y\to \mathsf{A}_{P},\theta\colon P\to Q,b\colon X\to \mathsf{A}_{Q})
    \end{equation*}
    such that $\theta$ is injective and that the underlying morphism $\ul{X}\to \ul{Y}\times_{\ul{\mathsf{A}_{P}}}\ul{\mathsf{A}_{Q}}$ is flat (that is the original definition of log flatness due to Kato).
\end{prop}

Due to the original work of Kato, neat charts and log flatness are related as follows.
(For references, see \cite[Theorem 4.15]{MR2032986} or \cite[Chapter \Rnum{4}, Theorem 4.1.7]{MR3838359}.)
\begin{thm}[{[Kato's unpublished work]}]\label{thm:flat_neat}
    Let $f\colon X\to Y$ be a log flat morphism of fine log schemes.
    Assume that the underlying morphism $\ul{f}\colon \ul{X}\to \ul{Y}$ is locally of finite presentation.
    Then, for any chart
        \begin{equation*}
        (a\colon Y\to \mathsf{A}_{P},\theta\colon P\to Q,b\colon X\to \mathsf{A}_{Q})
    \end{equation*}
    which is neat at $x$, the underlying morphism $\ul{X}\to \ul{Y}\times_{\ul{\mathsf{A}_{P}}} \ul{\mathsf{A}_{Q}}$ is flat at $x$ (or equivalently, in some neighborhood of x).
\end{thm}

In this paper, we need a variant of Theorem \ref{thm:flat_neat}.

\begin{prop}\label{prop:flat_neat_noeth}
    Let $f\colon X\to Y$ be a log flat morphism of fine log schemes.
    Assume that the underlying schemes $\ul{X}$ and $\ul{Y}$ are locally Noetherian.
    Then, for any fine chart
        \begin{equation*}
        (a\colon Y\to \mathsf{A}_{P},\theta\colon P\to Q,b\colon X\to \mathsf{A}_{Q})
    \end{equation*}
    which is neat at $x\in X$, the underlying morphism $\ul{X}\to \ul{Y}\times_{\ul{\mathsf{A}_{P}}} \ul{\mathsf{A}_{Q}}$ is flat at $x$.
\end{prop}
\begin{proof}
    The proof can be done by exactly following the proof of \cite[Chapter \Rnum{4}, Theorem 4.1.7]{MR3838359}.
    Take a fine chart $(a,\theta,b)$ which is neat at $x$.
    Let $b_{\theta}\colon X\to Y_{\theta}=Y\times_{\mathsf{A}_{P}}\mathsf{A}_{Q}$ the morphism induced by $f$ and $b$.
    Write $\mathsf{A}_{Q/P}^{*}$ for the group scheme $\Spec(\mathbb{Z}[\cok(\theta^{\gp})])$ with the trivial log structure.
    Consider the morphism $m_{Q/P}\colon \mathsf{A}_{Q}\times \mathsf{A}_{Q/P}^{*}\to \mathsf{A}_{Q}$ induced by the morphism $Q\to Q\oplus \cok(\theta^{\gp})$ sending $x$ to $(x,\bar{x})$.
    Then we have the diagram
    \begin{equation*}
        \begin{tikzpicture}[auto]
            \node (x) at (0,2) {$\ul{X}\times \mathsf{A}_{Q/P}^{*}$};
            \node (y) at (4,2) {$\ul{Y_{\theta}}\times \mathsf{A}_{Q/P}^{*}$};
            \node (s) at (2,0) {$\ul{Y_{\theta}},$};
            \draw[->] (x) to node {$\scriptstyle \ul{b_{\theta}}\times \id$} (y);
            \draw[->] (x) to node[swap] {$\scriptstyle \ul{m_{\theta}}$} (s);
            \draw[->] (y) to node {$\scriptstyle \ul{m_{Y}}$} (s);
        \end{tikzpicture}
    \end{equation*}
    where $m_{Y}\colon Y_{\theta}\times \mathsf{A}_{Q/P}^{*}\to Y_{\theta}$ is the morphism induced by $m_{Q/P}$ and $m_{\theta}\colon X\times \mathsf{A}_{Q/P}^{*}\to Y_{\theta}$ is the morphism induced by the action of $\mathsf{A}_{Q/P}^{*}$ and the morphism $X\to Y_{\theta}$.
    By the proof of \cite[Chapter \Rnum{4}, Theorem 4.1.7]{MR3838359}, the following hold:
    \begin{itemize}
        \item The morphism $\ul{m_{\theta}}$ is flat.
        \item The morphism $\ul{b_{\theta}^{+}}\colon \ul{X^{+}}=\ul{b_{\theta}}^{-1}(\ul{Y_{\theta}^{+}})\to \ul{Y_{\theta}^{+}}$ is flat, where $Y_{\theta}^{+}$ is the pullback of $Y_{\theta}\to \mathsf{A}_{Q}$ along the strict closed immersion $\mathsf{A}_{(Q,Q^{+})}=\Spec(\mathbb{Z}[Q]/\mathbb{Z}[Q^{+}])\hookrightarrow \mathsf{A}_{Q}$ induced by $Q\to \mathbb{Z}[Q]\to \mathbb{Z}[Q]/\mathbb{Z}[Q^{+}]$.
            Here, $\mathbb{Z}[Q^{+}]$ is the ideal generated by the set $Q^{+}$ of non-invertible elements of $Q$.
    \end{itemize}
    Note that we may assume that $b_{\bar{x}}\colon Q\to \mathcal{M}_{X,\bar{x}}$ induces $b_{\bar{x}}^{-1}(\mathcal{M}_{X,\bar{x}}^{\times})\cong \mathcal{M}_{X,\bar{x}}^{\times}$ by \cite[Chapter \Rnum{2}, Theorem 2.4.4]{MR3838359}.
    Hence the morphism $\beta\colon \mathbb{Z}[Q]\to \mathcal{O}_{X,\bar{x}}$ induced by $Q\to \mathcal{M}_{X,\bar{x}}$ and $\mathcal{M}_{X,\bar{x}}\to \mathcal{O}_{X,\bar{x}}$ satisfies $\mathbb{Z}[Q^{+}]\subset \beta^{-1}(\mathfrak{m}_{\bar{x}})$, which implies that $x$ belongs to $\ul{X^{+}}$.
    Therefore, due to \cite[Théorème (11.3.10)]{MR217086} (known as ``critère de platitude par fibres''), the morphism $\ul{b_{\theta}}\times \id$ is flat at every point $x'\in \ul{X}\times \mathsf{A}_{Q/P}^{*}$ lying over $x$.
    Since $\mathsf{A}_{Q/P}^{*}$ is faithfully flat over $\Spec(\mathbb{Z})$, we obtain the conclusion.
\end{proof}

\subsubsection{Log regularity and charts}

\begin{prop}[{\cite[Theorem 5.2.8]{MR4211089}}]\label{prop:logreg_chart}
    Let $f\colon X\to Y$ be a morphism of locally Noetherian fine log schemes.
    \begin{itemize}
        \item[\textup{(1)}] $f$ is log regular at $x\in X$ if and only if, for any fppf local fine chart $a\colon Y\to \mathsf{A}_{P}$, there exists an étale local fine chart
            \begin{equation*}
                (a\colon Y\to \mathsf{A}_{P},\theta\colon P\to Q,b\colon X\to \mathsf{A}_{Q}),
            \end{equation*}
            for $f$ satisfying
            \begin{itemize}
                \item[\textup{(\rnum{1})}] $\theta\colon P\to Q$ is injective,
                \item[\textup{(\rnum{2})}] The order of the torsion subgroup of $\cok(\theta^{\gp})$ is invertible in $k(\bar{x})$,
                \item[\textup{(\rnum{3})}] The underlying morphism $\ul{X}\to \ul{Y}\times_{\ul{\mathsf{A}_{P}}}\ul{\mathsf{A}_{Q}}$ is regular at $x$.
            \end{itemize}
        \item[\textup{(2)}] If $f$ is log regular at $x$, then, for any fppf local fine chart
            \begin{equation*}
                (a\colon Y\to \mathsf{A}_{P},\theta\colon P\to Q,b\colon X\to \mathsf{A}_{Q})
            \end{equation*}
            which is neat at $x$, the underlying morphism $\ul{X}\to \ul{Y}\times_{\ul{\mathsf{A}_{P}}}\ul{\mathsf{A}_{Q}}$ is regular at $x$.
    \end{itemize}
\end{prop}

Now, we can compare the relative log regularity with the absolute log regularity in the sense of Kato (\cite{MR1296725}).
Let $X$ be a locally Noetherian fine saturated log scheme.
Recall that $X$ is \textit{log regular} at $x\in X$ if the following hold:
\begin{itemize}
    \item[(1)] $\mathcal{O}_{X,\bar{x}}/I_{\bar{x}}$ is regular, where $I_{\bar{x}}$ denotes the ideal of $\mathcal{O}_{X,\bar{x}}$ generated by the images of non-invertible elements of $\mathcal{M}_{X,\bar{x}}$.
    \item[(2)] The equality $\dim(\mathcal{O}_{X,\bar{x}})=\dim(\mathcal{O}_{X,\bar{x}}/I_{\bar{x}})+\rank(\overline{\mathcal{M}}_{X,\bar{x}}^{\gp})$.
\end{itemize}
We say that $X$ is \textit{log regular} if it is log regular at each point.
\begin{prop}[{cf.\ \cite[Section 5.3]{MR4211089}}]
    Let $f\colon X\to Y$ be a morphism of locally Noetherian fine saturated log schemes.
    \begin{itemize}
        \item[\textup{(1)}] Assume that $f$ is log regular at $x$. Then $X$ is log regular at $x$ if and only if $Y$ is log regular at $y=f(x)$.
        \item[\textup{(2)}] Assume that $Y$ is the spectrum of a perfect field $k$ with the trivial log structure. Then $f$ is log regular at $x$ if and only if $X$ is log regular at $x$.
    \end{itemize}
\end{prop}
\begin{proof}
    (1) Since $f$ is log regular at $x$, we have an étale local fine chart
    \begin{equation*}
        (a\colon Y\to \mathsf{A}_{P},\theta\colon P\to Q,b\colon X\to \mathsf{A}_{Q})
    \end{equation*}
    as in Proposition \ref{prop:logreg_chart} (1).
    Then $f\colon X\to Y$ factors into $X\to Y'\to Y$, where $Y'=Y\times_{\mathsf{A}_{P}}\mathsf{A}_{Q}$.
    We see that the morphism $Y'\to Y$ is log smooth by \cite[Chapter \Rnum{4}, Theorem 3.1.8]{MR3838359}.
    Let $y'$ be the image of $x$ in $Y'$.
    Then $Y'$ is log regular at $y'$ if and only if $Y$ is log regular at $y$ by \cite[Chapter \Rnum{4}, Proposition 3.5.3]{MR3838359}.
    Therefore we may assume that $f$ is strict.
    Since $f$ is strict, we have $\overline{\mathcal{M}}_{Y,\bar{y}}=\overline{\mathcal{M}}_{X,\bar{x}}$ and $I_{\bar{x}}=I_{\bar{y}}\mathcal{O}_{X,\bar{x}}$.
    As $\mathcal{O}_{Y,\bar{y}}\to \mathcal{O}_{X,\bar{x}}$ is regular (and hence flat), we have
    \begin{itemize}
        \item $\mathcal{O}_{Y,\bar{y}}/I_{\bar{y}}$ is regular if and only if $\mathcal{O}_{X,\bar{x}}/I_{\bar{x}}$ is regular,
        \item $\dim(\mathcal{O}_{X,\bar{x}})-\dim(\mathcal{O}_{Y,\bar{y}})=\dim(\mathcal{O}_{X,\bar{x}}/I_{\bar{x}})-\dim(\mathcal{O}_{Y,\bar{y}}/I_{\bar{y}})$,
    \end{itemize}
    which completes the proof.

    (2) The only if part follows from (1).
    Assume that $X$ is log regular at $x$.
    Since $X$ is saturated, by Proposition \ref{prop:exist_neat} (1), we can take an étale local fine chart $P$ around $x$ such that $P\cong \overline{\mathcal{M}}_{X,\bar{x}}$.
    Then $\mathcal{O}_{X,\bar{x}}/I_{\bar{x}}$ is regular and $k[P]\to \mathcal{O}_{X,\bar{x}}$ is flat by \cite[(6.2) Theorem]{MR1296725}.
    Hence $\ul{X}\to \ul{Y}\times \ul{\mathsf{A}_{P}}$ is regular at $x$.
\end{proof}

\section{Main results}

In this section, we prove Theorem \ref{thm:logRA} and Corollary \ref{cor:logKunz}.
Before that, we recall absolute and relative Frobenius morphisms of log schemes.

For a log scheme $X$ over $\mathbb{F}_{p}$, the absolute Frobenius morphism $F_{X}$ of $X$ is given by the absolute Frobenius morphism $F_{\ul{X}}$ of the underlying scheme $\ul{X}$ and the morphism
\begin{equation*}
    \mathcal{M}_{X}\to \mathcal{M}_{X}
\end{equation*}
sending a local section $a$ to $a^{p}$.

For an integral morphism $f\colon X\to Y$ of integral log schemes, the relative Frobenius morphism $F_{X/Y}\colon X\to X^{(p)}$ fitting into the following commutative diagram
\begin{equation*}
    \begin{tikzpicture}[auto]
        \node (a) at (-2,2) {$X$};
        \node (b) at (0,0) {$X$};
        \node (c) at (2,2) {$Y$};
        \node (d) at (2,0) {$Y,$};
        \node (F) at (0,2) {$X^{(p)}$};
        \draw[->] (a) to[bend right=30] node[swap] {$\scriptstyle F_{X}$} (b);
        \draw[->] (a) to[bend left=60] node {$\scriptstyle f$} (c);
        \draw[->] (b) to node[swap] {$\scriptstyle f$} (d);
        \draw[->] (c) to node {$\scriptstyle F_{Y}$} (d);
        \draw[->] (F) to (b);
        \draw[->] (F) to (c);
        \draw[->, dashed] (a) to node {$\scriptstyle F_{X/Y}$} (F);
    \end{tikzpicture}
\end{equation*}
where the square is cartesian in the category of log schemes.
Note that $X^{(p)}$ is also integral by integrality of $f\colon X\to Y$.

\begin{thm}\label{thm:logRA}
    Let $f\colon X\to Y$ be an integral morphism of locally Noetherian fine log schemes over $\mathbb{F}_{p}$ such that $\mathcal{M}_{X/Y}^{\gp}$ is $p$-torsion free.
    Then $f$ is log regular if and only if $F_{X/Y}\colon X\to X^{(p)}$ is log flat.
\end{thm}
\begin{proof}
    Since $\mathcal{M}_{X/Y}^{\gp}$ is $p$-torsion free, by Proposition \ref{prop:exist_neat} (2), there is an étale local fine chart
    \begin{equation*}
        (a\colon Y\to \mathsf{A}_{P},\theta\colon P\to Q,b\colon X\to \mathsf{A}_{Q})
    \end{equation*}
    for $f$ which is neat at $x\in X$.
    We claim that the chart
    \begin{equation}\label{eq:chart}
        (b\colon X\to \mathsf{A}_{Q},F_{Q/P}\colon Q^{(p)}\to Q,b^{(p)}\colon X^{(p)}\to \mathsf{A}_{Q^{(p)}})
    \end{equation}
    for $F_{X/Y}\colon X\to X^{(p)}$ is neat at $x$, where $Q^{(p)}$ is the pushout of $P\to Q$ along the Frobenius morphism $P\to P$ (taken in the category of monoids).
    Take an element $\alpha=\overline{(g,h)}\in (Q^{(p)})^{\gp}=P^{\gp}\oplus Q^{\gp}/\sim$, where $\sim$ is given by $(s^{p},1)\sim (1,\theta^{\gp}(s))$.
    Assume that $1=F_{Q/P}^{\gp}(\alpha)=\theta^{\gp}(g)h^p$.
    Since $C:=\cok(\theta^{\gp})\cong \mathcal{M}_{X/Y,\bar{x}}^{\gp}$ is $p$-torsion free, we have $\overline{h}=1$ in $C$.
    Then we can write $h=\theta^{\gp}(g_{0})$ for some $g_{0}\in P^{\gp}$ and we have $1=\theta^{\gp}(g\cdot g_{0}^{p})$.
    Since $\theta^{\gp}$ is injective, we have $1=g\cdot g_{0}^{p}$ and hence $\alpha=\overline{(g,h)}=\overline{(g_{0}^{-p},\theta^{\gp}(g_{0}))}=1$, which implies that $F_{Q/P}^{\gp}$ is injective.
    Moreover, we have
    \begin{equation*}
        \cok(F_{Q/P}^{\gp}) \cong C/C^{p} \cong \mathcal{M}_{X/X^{(p)},\bar{x}}^{\gp}.
    \end{equation*}
    Thus the chart \eqref{eq:chart} is neat at $x$.

    Take an étale local fine chart $(a,\theta,b)$ for $f$ which is neat at $x\in X$.
    Apply Lemma \ref{lem:meta_fib} below to the morphism $\ul{X}\to \ul{Y}\times_{\ul{\mathsf{A}_{P}}}\ul{\mathsf{A}_{Q}}$ and the diagram
    \begin{equation*}
        \begin{tikzpicture}[auto]
            \node (a') at (0,2) {$\ul{Y}$};
            \node (s') at (2,2) {$\ul{\mathsf{A}_{P}}$};
            \node (b') at (4,2) {$\ul{\mathsf{A}_{Q}}$};
            \node (a) at (0,0) {$\ul{Y}$};
            \node (s) at (2,0) {$\ul{\mathsf{A}_{P}}$};
            \node (b) at (4,0) {$\ul{\mathsf{A}_{Q}},$};
            \draw[->] (a') to (s');
            \draw[->] (b') to (s');
            \draw[->] (a) to (s);
            \draw[->] (b) to (s);
            \draw[->] (a') to (a);
            \draw[->] (s') to (s);
            \draw[->] (b') to (b);
        \end{tikzpicture}
    \end{equation*}
    where each of the vertical arrows is the absolute Frobenius morphism of each object.
    Then we see that the following diagram is cartesian:
    \begin{equation*}
        \begin{tikzpicture}[auto]
            \node (f) at (0,2) {$\ul{X^{(p)}}\times_{\ul{\mathsf{A}_{Q^{(p)}}}}\ul{\mathsf{A}_{Q}}$};
            \node (x) at (0,0) {$\ul{X}$};
            \node (p') at (4,2) {$\ul{Y}\times_{\ul{\mathsf{A}_{P}}}\ul{\mathsf{A}_{Q}}$};
            \node (p) at (4,0) {$\ul{Y}\times_{\ul{\mathsf{A}_{P}}}\ul{\mathsf{A}_{Q}}$};
            \draw[->] (f) to (x);
            \draw[->] (f) to (p');
            \draw[->] (x) to (p);
            \draw[->] (p') to (p);
        \end{tikzpicture}
    \end{equation*}
    Hence, by using the classical Radu--André theorem, we see that the following conditions are equivalent:
    \begin{itemize}
        \item[\textup{(1)}] $\ul{X}\to \ul{Y}\times_{\ul{\mathsf{A}_{P}}}\ul{\mathsf{A}_{Q}}$ is regular at $x$.
        \item[\textup{(2)}] $\ul{X}\to \ul{X^{(p)}}\times_{\ul{\mathsf{A}_{Q^{(p)}}}}\ul{\mathsf{A}_{Q}}$ is flat at $x$.
    \end{itemize}
    Hence the assertion follows from Propositions \ref{prop:logflat}, \ref{prop:flat_neat_noeth}, and \ref{prop:logreg_chart}.
\end{proof}

\begin{lem}\label{lem:meta_fib}
    Let $\mathcal{C}$ be a category admitting fibered products.
    Consider the following diagrams in $\mathcal{C}$
    \begin{equation*}
        \begin{tikzpicture}[auto]
            \node (a') at (0,2) {$A'$};
            \node (s') at (2,2) {$S'$};
            \node (b') at (4,2) {$B'$};
            \node (a) at (0,0) {$A$};
            \node (s) at (2,0) {$S$};
            \node (b) at (4,0) {$B,$};
            \draw[->] (a') to (s');
            \draw[->] (b') to (s');
            \draw[->] (a) to (s);
            \draw[->] (b) to (s);
            \draw[->] (a') to (a);
            \draw[->] (s') to (s);
            \draw[->] (b') to (b);
        \end{tikzpicture}
    \end{equation*}
    where the squares are commutative.
    Then, for any morphism $X\to A\times_{S}B$, the following square is cartesian:
    \begin{equation*}
        \begin{tikzpicture}[auto]
            \node (f) at (0,2) {$(X\times_{A}A')\times_{(B\times_{S}S')}B'$};
            \node (x) at (0,0) {$X$};
            \node (p') at (4,2) {$A'\times_{S'}B'$};
            \node (p) at (4,0) {$A\times_{S}B$};
            \draw[->] (f) to (x);
            \draw[->] (f) to (p');
            \draw[->] (x) to (p);
            \draw[->] (p') to (p);
        \end{tikzpicture}
    \end{equation*}
\end{lem}
\begin{proof}
    By Yoneda's lemma, we may assume that $\mathcal{C}$ is the category of sets.
    Then the proof is straightforward.
\end{proof}

\begin{cor}\label{cor:logKunz}
    Let $X$ be a locally Noetherian fine log scheme over $\mathbb{F}_{p}$ such that $\overline{\mathcal{M}}_{X}^{\gp}$ is $p$-torsion free.
    Then $X\to \Spec(\mathbb{F}_{p})$ is log regular if and only if $F_{X}\colon X\to X$ is log flat.
    Here, we endow $\Spec(\mathbb{F}_{p})$ with the trivial log structure.

    In particular, if $X$ is saturated, then $X$ is log regular if and only if $F_{X}$ is log flat.
\end{cor}

\begin{rem}
    
    \indent

    \begin{itemize}
        \item[(1)] By Proposition \ref{prop:n-sat}, the condition that $\mathcal{M}_{X/Y}^{\gp}$ is $p$-torsion free in Theorem \ref{thm:logRA} is satisfied when $X$ and $f$ are $p$-saturated.
        \item[(2)] The condition that $\overline{\mathcal{M}}_{X}^{\gp}$ is $p$-torsion free in Corollary \ref{cor:logKunz} is satisfied when $X$ is $p$-saturated.
    \end{itemize}
\end{rem}


\end{document}